\documentclass[smallextended]{svjour3}       
\smartqed  
\usepackage{graphicx}
\usepackage{amsmath,amssymb}
\usepackage{braket}
\usepackage{bbm}

\usepackage{enumerate}
\newenvironment{thm}{\begin{theorem}}{\end{theorem}}

\newenvironment{lemm}{\begin{lemma}}{\end{lemma}}

\newenvironment{defi}{\begin{definition}}{\end{definition}}

\newcommand{\defarrow}{\stackrel{\mathrm{def.}}{\Leftrightarrow}}

\newcommand{\realn}{\mathbb{R}}

\newcommand{\pprime}{{\prime \prime}}
\newcommand{\aast}{{\ast \ast}}

\newcommand{\iin}{{i \in I}}
\newcommand{\jin}{{j \in J}}

\newcommand{\wsto}{\xrightarrow{\mathrm{weakly}\ast}}

\newcommand{\ikj}{{i(j)}}

\newcommand{\E}{\mathcal{E}}

\begin{document}

\title{Compatibility of any pair of $2$-outcome measurements characterizes
the Choquet simplex\thanks{This work was supported by Cross-ministerial 
Strategic Innovation Promotion Program (SIP)
(Council for Science, Technology and Innovation (CSTI)).}
}

\titlerunning{Compatibility of any measurements characterizes
the Choquet simplex}        

\author{Yui Kuramochi}


\institute{Yui Kuramochi \at
             Photon Science Center, Graduate School of Engineering,
             The University of Tokyo, 7-3-1 Hongo, Bunkyo-ku, Tokyo 113-8656, Japan \\
             Tel.: +81-03-5841-0859\\
              \email{kuramochi@qi.t.u-tokyo.ac.jp}           
}


\maketitle

\begin{abstract}
For a compact convex subset $K $ of a locally convex Hausdorff space, 
a measurement on $A(K)$ is 
a finite family of positive elements in $A(K)$ normalized to 
the unit constant $1_K , $
where $A(K)$ denotes the set of continuous real affine functionals on $K .$
It is proved that a compact convex set $K$ is a Choquet simplex 
if and only if any pair of $2$-outcome measurements are compatible,
i.e.\ the measurements are given as the marginals of a single measurement.
This generalizes the finite-dimensional result of
[Pl\'avala M 2016 Phys.\ Rev.\ A \textbf{94}, 042108]
obtained in the context of the foundations of quantum theory.
\end{abstract}

\section{Introduction and main results} \label{sec:intro}
Let $K$ be a compact convex subset of a locally convex Hausforff space
over the reals $\realn$
and let $A(K)$ denote the set of continuous real affine functionals on $K,$
where $A(K)$ is a Banach space under the supremum norm $\| f \| :=
\sup_{x\in K} |f(x)| $ $(f \in A(K)).$
(Throughout the paper all the linear spaces are over $\realn.$)
A measurement on $A(K)$ is a finite sequence 
$(f_k)_{k=1}^m \in A(K)^m $
for some integer $m$ such that $f_k \geq 0 $ $(1 \leq k \leq m)$
and $\sum_{k=1}^m f_k = 1_K ,$
where $1_K \equiv 1 $ is the unit constant functional.
A measurement belonging to $A(K)^m$ is called $m$-outcome.
Two measurements $(f_k)_{k=1}^m$ and $(g_j)_{j=1}^n$ on $A(K)$ are called 
\textit{compatible} (or \textit{jointly measurable}) 
if there exists a family $(h_{k,j})_{1\leq k \leq m , 1 \leq j \leq n}$
in $A(K) $
such that 
\[
	h_{k,j} \geq 0 ,
	\quad
	f_k = \sum_{j^\prime=1}^n h_{k,j^\prime} ,
	\quad
	g_j = \sum_{k^\prime =1}^m h_{k^\prime,j} \quad( 1 \leq k \leq m , \, 1 \leq j\leq n) .
\]
In \cite{PhysRevA.94.042108} Pl\'avala showed that 
for finite-dimensional $K ,$ 
$K$ is a simplex 
if and only if
any pair of $2$-outcome measurements are compatible.
In the physical context, a compact convex set $K$ corresponds to a state space
of a general physical system and Pl\'avala\rq{}s result indicates
that incompatibility of measurements~\cite{1751-8121-49-12-123001}
characterizes non-classicality of a physical system.
The proof in \cite{PhysRevA.94.042108} depends on the notion of the maximal face
and is not straightforwardly applicable to infinite-dimensional compact sets.
The purpose of this paper is to generalize this result 
to an arbitrary compact convex set $K $ (Theorems~\ref{thm:main1} and \ref{thm:main2}).

The first main result of this paper is the following theorem.
\begin{thm} \label{thm:main1}
Let $K$ be a compact convex subset of a locally convex Hausdorff space.
Then $K$ is a Choquet simplex if and only if 
any pair $(f_k)_{k=1}^2$ and $(g_j)_{j=1}^2$ of $2$-outcome measurements on $A(K)$ 
are compatible.
\end{thm}
Here a compact convex set $K$ is called a Choquet simplex if
the Banach dual $A(K)^\ast$ equipped with the dual positive cone 
\[
	A(K)^\ast_+ = \set{\psi \in A(K)^\ast | \psi (f) \geq 0 \, (\forall f \in A(K)_+)}
\]
is a lattice, where $ A(K)_+ := \set{f \in A(K) | f(x)\geq 0 \, (\forall x \in K)} .$

A Banach space $E$ is said to have a Banach predual $E_\ast$ if $E_\ast$ is 
a Banach space such that its Banach dual $(E_\ast)^\ast$ is isometrically isomorphic
to $E .$ We can and do take such $E_\ast ,$ if exists, as a closed linear subspace of 
the Banach dual $E^\ast  .$  
The second main result of this paper is then the following theorem.
\begin{thm} \label{thm:main2}
Let $K$ be a compact convex subset of a locally convex Hausdorff space.
Suppose that $A(K)$ has a Banach predual.
Then the following conditions are equivalent.
\begin{enumerate}[(i)]
\item
$K$ is a Bauer simplex.
\item
$K$ is a Choquet simplex.
\item
Any pair $(f_k)_{k=1}^2$ and $(g_j)_{j=1}^2$ of $2$-outcome measurements on $A(K)$
are compatible.
\end{enumerate}
\end{thm}
Here $K$ is a Bauer simplex if the set $\partial_\mathrm{e} K$ of 
extremal points of $K$ is compact and any point in $K$ is the barycenter of 
a unique boundary measure \cite{alfsen1971compact}.

The above theorems are rephrased in terms of order-unit Banach space by using 
the following one-to-one correspondence 
(\cite{alfsen1971compact}, Sections~II.1 and II.2).
A pair $(E,e)$ of an Archimedean ordered linear space $E$ and 
an order unit $e \in E$ is called an order-unit Banach space
if the order-unit norm 
$\|\cdot \|$
on $E$ induced by $e$ is complete.
We write by $E_+$ the positive cone of $E .$
For each order-unit Banach space $(E, e)$ its dual $E^\ast$ is 
a base-normed ordered linear space with the positive cone
$E_+^\ast := \set{\phi \in E^\ast | \phi (a) \geq 0 \, (a \in E_+)}$
and the base
$S(E) := \set{\phi \in E_+^\ast | \| \phi \| = \phi (e) =1 } .$
Each element of $S(E)$ is called a state on $E .$
$S(E)$ is a weakly$\ast$ compact convex subset of $E^\ast$ and 
the order-unit Banach space $(A(S(E)) , 1_{S(E)} )$ is shown to be 
order and isometrically isomorphic to $(E,e) .$
Conversely for any compact convex subset $K$ of a locally convex Hausdorff space,
$(A(K) , 1_K)$ is an order-unit Banach space and 
$K$ is continuously affine isomorphic to the set $S(A(K))$ of states on $A(K) .$

We also define the ($m$-outcome) measurement on an order-unit Banach space
$(E, e)$ as a finite-sequence (belonging to $E^m$) of positive elements of $E$
normalized to $e.$ 
The compatibility of measurements is defined in the same way as 
that of measurements on $A(K) .$
By this correspondence we can readily see that
Theorems~\ref{thm:main1} and \ref{thm:main2} are respectively equivalent to

\begin{thm}
\label{thm:main1p}
Let $(E,e)$ be an order-unit Banach space.
Then $E^\ast$ is a lattice
if and only if 
any pair of $2$-outcome measurements $(a_j)_{j=1}^2$ and $(b_k)_{k=1}^2$ on $(E,e)$
are compatible.
\end{thm}

\begin{thm}
\label{thm:main2p}
Let $(E,e)$ be an order-unit Banach space.
Suppose that $E$ has a Banach predual $E_\ast . $
Then the following conditions are equivalent.
\begin{enumerate}[(i)]
\item \label{i1}
$S(E)$ is a Bauer simplex.
\item \label{i2}
$E^\ast$ is a lattice.
\item \label{i3}
Any pair of $2$-outcome measurements $(a_j)_{j=1}^2$ and $(b_k)_{k=1}^2$ on $(E,e)$
are compatible.
\end{enumerate}
\end{thm}
In the rest of this paper, after introducing some preliminaries in Section~\ref{sec:prel}, 
we prove in Section~\ref{sec:proof}
the main results Theorems~\ref{thm:main1p} and \ref{thm:main2p}.

\section{Preliminary} \label{sec:prel}
This section reviews necessary results of order-unit Banach spaces with 
Banach preduals and those of simplexes. 
The reader is referred to \cite{alfsen1971compact} and 
\cite{Ellis1964duality,Olubummo1999} for the 
complete proofs of the facts given in this section.

Let $(E,e)$ be an order-unit Banach space with a Banach predual $E_\ast .$
Then the predual $E_\ast$ endowed with the predual positive cone
\[
	E_{\ast +} 
	:=
	\set{\psi \in E_\ast | \psi (a) \geq 0 \, (\forall a \in E_+)}
\]
has the base 
\[
	S_\ast (E)
	:=
	\set{\psi \in E_{\ast +} | \psi(e) =1 }
	=E_{\ast} \cap S(E)
\]
and the base norm on $E_\ast$ induced by $S_\ast (E)$ coincides with the 
predual norm 
\[
\| \psi \| = \sup_{a \in E \colon \| a \| \leq 1} | \psi (a) |
\quad
(\psi \in E_\ast),
\]
which is the norm on $E^\ast$ restricted to $E_\ast $
(\cite{Ellis1964duality}, Theorem~6; \cite{Olubummo1999}).
Moreover the positive cone $E_+$ is weakly$\ast$ closed
(i.e.\ $\sigma (E,E_\ast)$-closed) 
(\cite{Olubummo1999}, in the proof).
Therefore by the bipolar theorem $E_+$ is the dual cone of $E_{\ast +} ,$ i.e.\
\begin{equation}
	E_+ 
	=\set{a \in E | \psi (a) \geq 0 \, (\forall \psi \in E_{\ast +})} .
	\notag
\end{equation}
By the Banach-Alaoglu theorem, the unit ball of $E$ is weakly$\ast$ compact
and hence so is any interval 
\[
	[a,b] 
	:=
	\set{x \in E | a \leq x \leq b}
	\quad
	(a,b \in E).
\]

We have seen that a Banach predual $E_\ast$ of an order-unit Banach space $E$ is 
a base-normed Banach space and the positive cone of $E$ is the dual cone of that of
$E_\ast .$
We can conversely show that the Banach dual $(E_\ast)^\ast $ of a 
base-normed Banach space $E_\ast$ 
is an order-unit Banach space, where the positive cone 
of $(E_\ast)^\ast$ is the dual cone of that of $E_\ast $ 
(\cite{Ellis1964duality}; \cite{alfsen1971compact}, II.1.15).
Thus we can start either from an order-unit Banach space $E$ with a Banach predual
or from a base-normed space $E_\ast ,$
and these definitions result in effectively the same concept.
We exploit the former definition since the main 
Theorems~\ref{thm:main2} and \ref{thm:main2p} are stated in terms of this.

A net $(x_i)_{\iin}$ in an ordered linear space $E$ is called bounded if $x_i \leq a$ 
$(\forall \iin)$ for some $a \in E$
and increasing if $i \leq i^\prime$ implies $x_i \leq x_{i^\prime} $
$(i,i^\prime \in I).$
We denote the supremum of the subset $\set{x_i | \iin}$ of $E$
by $\sup_\iin x_i  $ if it exists.

The following lemma will be used in the proof of Theorem~\ref{thm:main2p}.
\begin{lemm} \label{lemm:sup}
Let $(E, e)$ be an order-unit Banach space with a Banach predual
$E_\ast .$
Then any bounded increasing net $(x_i)_\iin$ in $E$ is weakly$\ast$ convergent
to the supremum $\sup_\iin x_i .$
\end{lemm}
\begin{proof}
Take an element $a \in E$ such that $x_i \leq a$
$(\iin ) .$ For a fixed $i_0 \in I$ we have $x_i \in [x_{i_0} , a]$ eventually
and by the compactness of $ [x_{i_0} , a]$ there exists a subnet 
$(x_\ikj)_\jin$ weakly$\ast$ converging to some $x \in  [x_{i_0} , a] .$ 
Since $x_i \leq x_\ikj$ eventually for each $\iin ,$
the weak$\ast$ closedness of the positive cone $E_+$ implies 
$x_i \leq x$ for all $\iin .$
Moreover if $x_i \leq b \in E$ $(\iin) ,$
then again by the weak$\ast$ closedness of $E_+$ we have $x \leq b .$
Therefore $x = \sup_{\iin} x_i .$ 
Take arbitrary $\psi \in E_{\ast +} .$
Then $(\psi (x_i))_\iin$ is an increasing net in $\realn$
and 
\[
	\psi(x) \geq \sup_\iin \psi(x_i) \geq \sup_\jin \psi (x_\ikj) = \psi (x) ,
\]
which implies $\psi (x_i) \uparrow \psi (x) .$
Since $E_\ast$ is the linear span of $E_{\ast +} ,$
this implies $x_i \wsto x .$
\qed \end{proof}

Let $(E,e)$ be an order-unit Banach space.
We define the linear order $\leq $ on the
double dual space $E^\aast$ corresponding to the double dual cone
\[
	E^\aast_+ := \set{x^\pprime \in E^\aast | \braket{x^\pprime , \psi} \geq 0 \,
	( \forall \psi \in E_+^\ast )} ,
\]
where we introduced the notation $\braket{x^\pprime , \psi} := x^\pprime (\psi)$
$(x^\pprime \in E^\aast , \psi \in E^\ast) .$
As usual we regard $E$ as a norm-closed subspace of $E^\aast.$
Then  by the bipolar theorem we have $E_+^\aast \cap E = E_+ , $ 
which implies that the order on $E^\aast$ restricted to $E$ coincides with 
the original order on $E .$
Moreover $(E^\aast , e)$ is an order-unit Banach space with the Banach predual
$E^\ast$ and $E_+^\aast$ is the weak$\ast$ closure 
of $E_+$ by the bipolar theorem.
Let us introduce the sets
\begin{gather*}
	\E (E) := \set{x \in E | 0 \leq x \leq e} ,
	\\
	\E (E^\aast) := \set{x^\pprime \in E^\aast | 0 \leq x^\pprime \leq e} .
\end{gather*}
Each element of $\E (E)$ or $\E (E^\aast)$ is called an \textit{effect}.
\begin{lemm} \label{lemm:dense}
Let $(E,e)$ be an order-unit Banach space.
Then $\E (E)$ is weakly$\ast$ dense in $\E (E^\aast) .$
\end{lemm}
\begin{proof}
Suppose that $\E (E)$ is not weakly$\ast$ dense in $\E (E^\aast) .$
Then by the Hahn-Banach separation theorem, there exist $\psi \in E^\ast$
and $x^\pprime \in \E (E^\aast)$
such that
\[
	\sup_{x\in \E(E)} \braket{x,\psi} 
	< \braket{x^\pprime , \psi} .
\]
Since the unit ball of $E$ coincides with
\[
	\set{-e + 2x | x \in \E (E)},
\]
we have
\[
	\| \psi \|
	= 
	\sup_{x\in \E (E)}
	\braket{-e+2x , \psi}
	< 
	\braket{-e + 2x^\pprime , \psi}
	\leq \| -e + 2x^\pprime \| \| \psi \|
	\leq \| \psi \| ,
\]
which is a contradiction.
\qed \end{proof}

The following well-known characterizations of the Choquet and Bauer 
simplexes are also necessary.

\begin{lemm}[\cite{alfsen1971compact}, II.3.1 and II.3.11] \label{lemm:choquet}
Let $(E,e)$ be an order-unit Banach space.
Then $E^\ast$ is a lattice if and only if $E$ has the Riesz decomposition 
property, i.e.\
\begin{align*}
	&0 \leq u \leq v_1 + v_2 , \quad 0 \leq  v_1 , v_2 
	\\ &\implies
	\exists u_1 , u_2\in E \colon
	0 \leq u_j \leq v_j \, (j=1,2) ,
	\quad
	u = u_1 + u_2 
\end{align*}
holds for any $u, v_1 , v_2 \in E .$
\end{lemm}

\begin{lemm}[\cite{alfsen1971compact}, II.4.1] \label{lemm:bauer}
Let $(E,e)$ be an order-unit Banach space.
Then $S(E)$ is a Bauer simplex if and only if $E$ is a lattice.
\end{lemm}

\section{Proofs of Theorems~\ref{thm:main1p} and \ref{thm:main2p}} \label{sec:proof}
In this section, we first prove Theorem~\ref{thm:main2p} 
and then Theorem~\ref{thm:main1p} by reducing to Theorem~\ref{thm:main2p}.

Let $(E,e) $ be an order-unit Banach space. 
For each effect $a \in \E ( E)$ there corresponds a $2$-outcome measurement 
$(a, e-a)$ on $(E,e) .$
Two effects $a,b \in E$ are said to be compatible if the corresponding 
$2$-outcome measurements $(a, e-a)$ and $(b ,e-b)$ are compatible.
The following lemma is immediate from the definition of the compatibility
(cf.\ \cite{PhysRevA.94.042108}, Proposition~7).
\begin{lemm} \label{lemm:2out}
Let $(E,e)$ be an order-unit Banach space.
Then effects $a,b\in \mathcal{E} (E)$ are compatible if and only if
there exists $c\in E$ such that
\begin{equation}
	0 , a+b-e \leq c \leq a,b .
	\notag
\end{equation}
\end{lemm}

We also introduce the following notion of orthogonality.
\begin{defi} \label{def:ortho}
Positive elements $a $ and $b$ in an ordered linear space $E$ is said to be 
\textit{orthogonal} if 
\[
	0 \leq c \leq a,b \implies c=0
\]
for any $c\in E .$ 
It readily follows that 
for positive elements $a,b \in E$ and positive reals $\alpha , \beta >0 ,$
$a$ and $b$ are orthogonal if and only if $\alpha a$ and $\beta b$ are orthogonal.
\end{defi}
Then we have
\begin{lemm} \label{lemm:ortho}
Let $(E,e)$ be an order-unit Banach space satisfying the condition~\eqref{i3}
of Theorem~\ref{thm:main2p}
and let $a,b \in E$ be orthogonal positive elements.
Suppose that $a \neq 0 .$
Then for any
$\psi \in S(E)$
with $\psi (a) = \|a\|   ,$
it holds that $\psi (b) =0.$
\end{lemm}
\begin{proof}
The statement is obvious when $b=0$ and we assume $b \neq 0 .$
Then since the effects $\| a \|^{-1} a  $ and $\| b \|^{-1} b$ are 
compatible, by Lemma~\ref{lemm:2out} there exists $c\in E$ such that
\[
	0 , \| a\|^{-1} a + \| b \|^{-1} b -e \leq c \leq \| a \|^{-1} a  , \| b \|^{-1} b .
\]
This implies 
\[
	0 \leq \min (\| a \| , \| b \|)c  \leq a, b
\]
and hence the orthogonality of $a$ and $b$ implies $c=0 .$
Therefore 
\[
	\| a\|^{-1}a + \| b\|^{-1}b -e \leq 0 .
\]
Thus for any state $\psi \in S(E)$ with $\psi (a) = \| a\|$ we have
\[
	0 \leq \| b\|^{-1} \psi (b) 
	\leq \psi (e - \|a\|^{-1} a  )
	=0 ,
\]
which implies $\psi(b) =0 .$
\qed \end{proof}

\noindent
\textit{Proof of Theorem~\ref{thm:main2p}.}

\eqref{i1}$\implies$\eqref{i2} is obvious.

\eqref{i2}$\implies$\eqref{i1}.
This can be proved in the same way as in Proposition~II.3.2 of \cite{alfsen1971compact}
by noting Lemma~\ref{lemm:bauer} and 
that the set 
\[
	\set{w \in E | x,y \leq w \leq z} = [x,z] \cap [y,z]
\]
for any $x,y,z\in E$ is weakly$\ast$ compact.

\eqref{i2}$\implies$\eqref{i3}.
Assume \eqref{i2} and  
take arbitrary $2$-outcome measurements 
$(a_1 , a_2) = (a_1, e-a_1)$ and $(b_1 , b_2 ) = (b_1 , e-b_1)  $ on 
$(E, e).$
Then by $a_1 \leq e = b_1 + b_2$ and the Riesz decomposition property of $E,$
there exist elements $c_{1,1} , c_{1,2} \in E$ such that
\[
	0 \leq c_{1,j} \leq b_j \, (j=1,2), \quad 
	a_1 = c_{1,1} + c_{1,2} .
\]
If we put $c_{2,j} := b_j - c_{1,j} \geq 0 ,$ we have 
\begin{gather*}
	b_j = c_{1,j} + c_{2,j} \quad (j=1,2),
	\\
	a_2 = e- a_1 = b_1 + b_2 - c_{1,1} - c_{1,2} = c_{2,1} + c_{2,2} .
\end{gather*}
Thus $(a_1 , a_2)$ and $(b_1 , b_2)$ are compatible.

\eqref{i3}$\implies$\eqref{i2}.
We assume \eqref{i3} and prove the Riesz decomposition property of $E.$
Take elements $u,v_1, v_2 \in E$ satisfying
\[
	0 \leq u \leq v_1 +  v_2 , \quad 0 \leq v_1, v_2 .
\]
Let
\[
	A:= \set{(x_1, x_2) \in E\times E | 0 \leq x_j \leq v_j \, (j=1,2) , \,  x_1 + x_2 \leq u  },
\]
which is non-empty by $(0,0) \in A ,$
and define a partial order $\leq $ on $A$ by
\[
	(x_1 , x_2 ) \leq (y_1, y_2)
	\, :\defarrow \,
	x_1 \leq y_1 \text{ and } x_2 \leq y_2 
	\quad
	((x_1 , x_2 ) , (y_1, y_2) \in A).
\] 
Let $D$ be a directed subset of $A$ and for each $\delta \in D$ 
we write as
$\delta = (x_1^\delta , x_2^\delta) .$
Then for each $j=1,2$ the net $(x_j^\delta)_{\delta \in D}$ in $E$
is increasing and upper bounded by $v_j .$
Thus by Lemma~\ref{lemm:sup} 
\[
x_j^\delta \wsto  \sup_{\delta \in D}  x_j^\delta =: x_j \in E.
\]
By the weak$\ast$ closedness of $E_+ ,$ we have $(x_1 , x_2) \in A$
and hence $(x_1 , x_2)$ is a supremum of $D$ in $A .$
Therefore by Zorn\rq{}s lemma there exists an element $(u_1 , u_2) \in A$
maximal with respect to the order $\leq $ on $A .$

We next prove that $u-u_1 - u_2 \geq 0$ and $v_1 - u_1 \geq 0$ are orthogonal.
Take $c \in E$ such that
\[
	0 \leq c \leq u-u_1 - u_2 ,v_1 - u_1 .
\]
Then this implies $(u_1 , u_2) \leq (u_1 +c , u_2) \in A$
and hence by the maximality of $(u_1 , u_2)$ we have $c=0 .$
Therefore $u-u_1 - u_2$ and $v_1 - u_1$ are orthogonal.
Similar proof also applies to the orthogonality of 
$u-u_1 - u_2$ and $v_2 - u_2 .$

Now assume $u-u_1 -u_2 \neq 0.$
We then take a state $\psi \in S(E)$ such that 
\[
\psi (u-u_1 -u_2) = \| u-u_1 -u_2 \| > 0 .
\]
Then by Lemma~\ref{lemm:ortho} we have
$
	\psi (v_1 - u_1) = \psi(v_2 - u_2) = 0
$
since $u-u_1 -u_2$ and $v_j - u_j $ are orthogonal $(j=1,2) .$
Hence, by noting $u \leq v_1 + v_2,$ we have
\[
	0 <  \psi (u-u_1 -u_2)
	=
	\psi (u - v_1 - v_2) 
	\leq 0 ,
\]
which is a contradiction.
Therefore $u=u_1 +u_2 ,$
which proves the Riesz decomposition property of $E .$
\qed

\noindent
\textit{Proof of Theorem~\ref{thm:main1p}.}
The proof of \eqref{i2}$\implies$\eqref{i3} in Theorem~\ref{thm:main2p} 
also applies to 
the \lq\lq{}only if\rq\rq{} part of the statement.
To establish the \lq\lq{}if\rq\rq{} part, we assume that any pair of 
$2$-outcome measurements on $(E,e)$ is compatible
and show that $E^\ast$ is a lattice. 

We first prove that any $2$-outcome measurements on $(E^\aast ,e)$ is 
compatible.
Since each $2$-outcome measurements corresponds to an effect by
\[
	\mathcal{E} (E^\aast) \ni x^\pprime 
	\mapsto 
	(x^\pprime , e-x^\pprime)
	,
\]
we have only to prove that any effects $a^\pprime , b^\pprime \in \E(E^\aast)$
are compatible.
By Lemma~\ref{lemm:dense} there exist nets $(a_i)_\iin$ and $(b_i)_\iin $
in $\E (E)$ such that $a_i \wsto a^\pprime$ and $b_i \wsto b^\pprime .$
Since the effects $a_i$ and $b_i$ are compatible for each $\iin ,$
there exists $c_i \in E$ such that
\[
	0, a_i + b_i - e \leq c_i \leq a_i , b_i .
\]
Since $c_i$ is bounded, the Banach-Alaoglu theorem implies that 
there exists a subnet of $(c_i)_\iin$ weakly$\ast$ converging to some 
$c^\pprime \in E^\aast .$
By the weak$\ast$ closedness of the positive cone $E_+^\aast ,$
we have 
\[
	0 , a^\pprime + b^\pprime - e \leq c^\pprime \leq a^\pprime , b^\pprime ,
\]
which implies the compatibility of the effects $a^\pprime$ and $b^\pprime.$

Now by Theorem~\ref{thm:main2p},
$E^\aast$ is a lattice and hence an AM-space and 
by Corollary~1 of \cite{Ellis1964duality}
$E^\ast$ is a AL-space and so is a lattice,
which completes the proof.
\qed



\begin{thebibliography}{1}
\bibitem{PhysRevA.94.042108}
Pl\'avala, M.: All measurements in a probabilistic theory are compatible if and
  only if the state space is a simplex.
\newblock Phys. Rev. A \textbf{94}, 042108 (2016).
\end{thebibliography}

\end{document}